\documentclass[runningheads,a4paper,envcountsame,envcountsect]{llncs}

\usepackage[utf8]{inputenc}
\usepackage{xr}
\let\saveendproof=\endproof
\def\endproof{\qed\saveendproof}

\usepackage{pdfpages}
\usepackage{pdflscape} 

\usepackage{booktabs}
\usepackage{amssymb}
\usepackage{graphicx}
\usepackage{epstopdf}
\usepackage[hyphens]{url}

\usepackage{multirow}

\usepackage{tikz}
\definecolor{mediumspringgreen}{rgb}{0.0, 0.98039215, 0.60392156}

\usepackage{pgfplots} 

\usetikzlibrary{positioning}
\def\visible<#1>{}  
\usepackage[utf8]{inputenc}

\usepackage[english]{babel}
\usepackage{amsfonts}
\usepackage{amsmath}
\usepackage{latexsym}
\usepackage{color}
\usepackage{subfigure}
\usepackage{enumerate}

\usepackage{hyperref}  

\usepackage{ifpdf}
\usepackage{listings}
\DeclareMathOperator    \conv           {conv}

\DeclareMathOperator    \relint         {rel\,int}

\DeclareMathOperator    \verts          {vert}

\newcommand{\old}[1]{{}}

\newcommand{\bb}{\mathbb}

\newcommand{\R}{\bb R}
\newcommand{\Q}{\bb Q}
\newcommand{\Z}{\bb Z}

\newcommand\st{\mid}

\renewcommand{\P}{\mathcal{P}}

\def\st{\mid}

\newenvironment{psmallmatrixbig}{\bigl(\smallmatrix}{\endsmallmatrix\bigr)}
\newcommand\InlineFrac[2]{#1/#2}  
\newcommand\ColVec[3][\relax]
{
  \ifx#1\relax
  \bgroup\let\frac=\InlineFrac\begin{psmallmatrixbig}#2\vphantom{/}\\#3\vphantom{/}\end{psmallmatrixbig}\egroup
  \else
  \bgroup\let\frac=\InlineFrac\begin{psmallmatrixbig}\ifx#200\else#2/#1\fi\\\ifx#300\else#3/#1\fi\end{psmallmatrixbig}\egroup
  \fi
}

\makeatletter
\renewcommand{\pod}[1]
{\allowbreak\mathchoice{\mkern18mu}{\mkern8mu}{\mkern8mu}{\mkern8mu}(#1)}
\makeatother

\chardef\Myunderscore=`\_
\newcommand\underscore{\Myunderscore\allowbreak}
\let\_=\underscore

\newcommand\githubsearchurl{https://github.com/mkoeppe/infinite-group-relaxation-code/search}
\pgfkeyssetvalue{/sagefunc/bccz_counterexample}{\href{\githubsearchurl?q=\%22def+bccz_counterexample(\%22}{\sage{bccz\underscore{}counterexample}}}
\pgfkeyssetvalue{/sagefunc/bcdsp_arbitrary_slope}{\href{\githubsearchurl?q=\%22def+bcdsp_arbitrary_slope(\%22}{\sage{bcdsp\underscore{}arbitrary\underscore{}slope}}}
\pgfkeyssetvalue{/sagefunc/bhk_gmi_irrational}{\href{\githubsearchurl?q=\%22def+bhk_gmi_irrational(\%22}{\sage{bhk\underscore{}gmi\underscore{}irrational}}}
\pgfkeyssetvalue{/sagefunc/bhk_irrational}{\href{\githubsearchurl?q=\%22def+bhk_irrational(\%22}{\sage{bhk\underscore{}irrational}}}
\pgfkeyssetvalue{/sagefunc/bhk_slant_irrational}{\href{\githubsearchurl?q=\%22def+bhk_slant_irrational(\%22}{\sage{bhk\underscore{}slant\underscore{}irrational}}}
\pgfkeyssetvalue{/sagefunc/chen_4_slope}{\href{\githubsearchurl?q=\%22def+chen_4_slope(\%22}{\sage{chen\underscore{}4\underscore{}slope}}}
\pgfkeyssetvalue{/sagefunc/dg_2_step_mir}{\href{\githubsearchurl?q=\%22def+dg_2_step_mir(\%22}{\sage{dg\underscore{}2\underscore{}step\underscore{}mir}}}
\pgfkeyssetvalue{/sagefunc/dg_2_step_mir_limit}{\href{\githubsearchurl?q=\%22def+dg_2_step_mir_limit(\%22}{\sage{dg\underscore{}2\underscore{}step\underscore{}mir\underscore{}limit}}}
\pgfkeyssetvalue{/sagefunc/drlm_2_slope_limit}{\href{\githubsearchurl?q=\%22def+drlm_2_slope_limit(\%22}{\sage{drlm\underscore{}2\underscore{}slope\underscore{}limit}}}
\pgfkeyssetvalue{/sagefunc/drlm_3_slope_limit}{\href{\githubsearchurl?q=\%22def+drlm_3_slope_limit(\%22}{\sage{drlm\underscore{}3\underscore{}slope\underscore{}limit}}}
\pgfkeyssetvalue{/sagefunc/drlm_backward_3_slope}{\href{\githubsearchurl?q=\%22def+drlm_backward_3_slope(\%22}{\sage{drlm\underscore{}backward\underscore{}3\underscore{}slope}}}
\pgfkeyssetvalue{/sagefunc/gj_2_slope}{\href{\githubsearchurl?q=\%22def+gj_2_slope(\%22}{\sage{gj\underscore{}2\underscore{}slope}}}
\pgfkeyssetvalue{/sagefunc/gj_2_slope_repeat}{\href{\githubsearchurl?q=\%22def+gj_2_slope_repeat(\%22}{\sage{gj\underscore{}2\underscore{}slope\underscore{}repeat}}}
\pgfkeyssetvalue{/sagefunc/gj_forward_3_slope}{\href{\githubsearchurl?q=\%22def+gj_forward_3_slope(\%22}{\sage{gj\underscore{}forward\underscore{}3\underscore{}slope}}}
\pgfkeyssetvalue{/sagefunc/gmic}{\href{\githubsearchurl?q=\%22def+gmic(\%22}{\sage{gmic}}}
\pgfkeyssetvalue{/sagefunc/hildebrand_2_sided_discont_1_slope_1}{\href{\githubsearchurl?q=\%22def+hildebrand_2_sided_discont_1_slope_1(\%22}{\sage{hildebrand\underscore{}2\underscore{}sided\underscore{}discont\underscore{}1\underscore{}slope\underscore{}1}}}
\pgfkeyssetvalue{/sagefunc/hildebrand_2_sided_discont_2_slope_1}{\href{\githubsearchurl?q=\%22def+hildebrand_2_sided_discont_2_slope_1(\%22}{\sage{hildebrand\underscore{}2\underscore{}sided\underscore{}discont\underscore{}2\underscore{}slope\underscore{}1}}}
\pgfkeyssetvalue{/sagefunc/hildebrand_5_slope_22_1}{\href{\githubsearchurl?q=\%22def+hildebrand_5_slope_22_1(\%22}{\sage{hildebrand\underscore{}5\underscore{}slope\underscore{}22\underscore{}1}}}
\pgfkeyssetvalue{/sagefunc/hildebrand_5_slope_24_1}{\href{\githubsearchurl?q=\%22def+hildebrand_5_slope_24_1(\%22}{\sage{hildebrand\underscore{}5\underscore{}slope\underscore{}24\underscore{}1}}}
\pgfkeyssetvalue{/sagefunc/hildebrand_5_slope_28_1}{\href{\githubsearchurl?q=\%22def+hildebrand_5_slope_28_1(\%22}{\sage{hildebrand\underscore{}5\underscore{}slope\underscore{}28\underscore{}1}}}
\pgfkeyssetvalue{/sagefunc/hildebrand_discont_3_slope_1}{\href{\githubsearchurl?q=\%22def+hildebrand_discont_3_slope_1(\%22}{\sage{hildebrand\underscore{}discont\underscore{}3\underscore{}slope\underscore{}1}}}
\pgfkeyssetvalue{/sagefunc/kf_n_step_mir}{\href{\githubsearchurl?q=\%22def+kf_n_step_mir(\%22}{\sage{kf\underscore{}n\underscore{}step\underscore{}mir}}}
\pgfkeyssetvalue{/sagefunc/kzh_10_slope_1}{\href{\githubsearchurl?q=\%22def+kzh_10_slope_1(\%22}{\sage{kzh\underscore{}10\underscore{}slope\underscore{}1}}}
\pgfkeyssetvalue{/sagefunc/kzh_28_slope_1}{\href{\githubsearchurl?q=\%22def+kzh_28_slope_1(\%22}{\sage{kzh\underscore{}28\underscore{}slope\underscore{}1}}}
\pgfkeyssetvalue{/sagefunc/kzh_28_slope_2}{\href{\githubsearchurl?q=\%22def+kzh_28_slope_2(\%22}{\sage{kzh\underscore{}28\underscore{}slope\underscore{}2}}}
\pgfkeyssetvalue{/sagefunc/kzh_3_slope_param_extreme_1}{\href{\githubsearchurl?q=\%22def+kzh_3_slope_param_extreme_1(\%22}{\sage{kzh\underscore{}3\underscore{}slope\underscore{}param\underscore{}extreme\underscore{}1}}}
\pgfkeyssetvalue{/sagefunc/kzh_3_slope_param_extreme_2}{\href{\githubsearchurl?q=\%22def+kzh_3_slope_param_extreme_2(\%22}{\sage{kzh\underscore{}3\underscore{}slope\underscore{}param\underscore{}extreme\underscore{}2}}}
\pgfkeyssetvalue{/sagefunc/kzh_4_slope_param_extreme_1}{\href{\githubsearchurl?q=\%22def+kzh_4_slope_param_extreme_1(\%22}{\sage{kzh\underscore{}4\underscore{}slope\underscore{}param\underscore{}extreme\underscore{}1}}}
\pgfkeyssetvalue{/sagefunc/kzh_5_slope_fulldim_1}{\href{\githubsearchurl?q=\%22def+kzh_5_slope_fulldim_1(\%22}{\sage{kzh\underscore{}5\underscore{}slope\underscore{}fulldim\underscore{}1}}}
\pgfkeyssetvalue{/sagefunc/kzh_5_slope_fulldim_2}{\href{\githubsearchurl?q=\%22def+kzh_5_slope_fulldim_2(\%22}{\sage{kzh\underscore{}5\underscore{}slope\underscore{}fulldim\underscore{}2}}}
\pgfkeyssetvalue{/sagefunc/kzh_5_slope_fulldim_3}{\href{\githubsearchurl?q=\%22def+kzh_5_slope_fulldim_3(\%22}{\sage{kzh\underscore{}5\underscore{}slope\underscore{}fulldim\underscore{}3}}}
\pgfkeyssetvalue{/sagefunc/kzh_5_slope_fulldim_4}{\href{\githubsearchurl?q=\%22def+kzh_5_slope_fulldim_4(\%22}{\sage{kzh\underscore{}5\underscore{}slope\underscore{}fulldim\underscore{}4}}}
\pgfkeyssetvalue{/sagefunc/kzh_5_slope_fulldim_5}{\href{\githubsearchurl?q=\%22def+kzh_5_slope_fulldim_5(\%22}{\sage{kzh\underscore{}5\underscore{}slope\underscore{}fulldim\underscore{}5}}}
\pgfkeyssetvalue{/sagefunc/kzh_5_slope_fulldim_covers_1}{\href{\githubsearchurl?q=\%22def+kzh_5_slope_fulldim_covers_1(\%22}{\sage{kzh\underscore{}5\underscore{}slope\underscore{}fulldim\underscore{}covers\underscore{}1}}}
\pgfkeyssetvalue{/sagefunc/kzh_5_slope_fulldim_covers_2}{\href{\githubsearchurl?q=\%22def+kzh_5_slope_fulldim_covers_2(\%22}{\sage{kzh\underscore{}5\underscore{}slope\underscore{}fulldim\underscore{}covers\underscore{}2}}}
\pgfkeyssetvalue{/sagefunc/kzh_5_slope_fulldim_covers_3}{\href{\githubsearchurl?q=\%22def+kzh_5_slope_fulldim_covers_3(\%22}{\sage{kzh\underscore{}5\underscore{}slope\underscore{}fulldim\underscore{}covers\underscore{}3}}}
\pgfkeyssetvalue{/sagefunc/kzh_5_slope_fulldim_covers_4}{\href{\githubsearchurl?q=\%22def+kzh_5_slope_fulldim_covers_4(\%22}{\sage{kzh\underscore{}5\underscore{}slope\underscore{}fulldim\underscore{}covers\underscore{}4}}}
\pgfkeyssetvalue{/sagefunc/kzh_5_slope_fulldim_covers_5}{\href{\githubsearchurl?q=\%22def+kzh_5_slope_fulldim_covers_5(\%22}{\sage{kzh\underscore{}5\underscore{}slope\underscore{}fulldim\underscore{}covers\underscore{}5}}}
\pgfkeyssetvalue{/sagefunc/kzh_5_slope_fulldim_covers_6}{\href{\githubsearchurl?q=\%22def+kzh_5_slope_fulldim_covers_6(\%22}{\sage{kzh\underscore{}5\underscore{}slope\underscore{}fulldim\underscore{}covers\underscore{}6}}}
\pgfkeyssetvalue{/sagefunc/kzh_5_slope_q22_f10_1}{\href{\githubsearchurl?q=\%22def+kzh_5_slope_q22_f10_1(\%22}{\sage{kzh\underscore{}5\underscore{}slope\underscore{}q22\underscore{}f10\underscore{}1}}}
\pgfkeyssetvalue{/sagefunc/kzh_5_slope_q22_f10_2}{\href{\githubsearchurl?q=\%22def+kzh_5_slope_q22_f10_2(\%22}{\sage{kzh\underscore{}5\underscore{}slope\underscore{}q22\underscore{}f10\underscore{}2}}}
\pgfkeyssetvalue{/sagefunc/kzh_5_slope_q22_f10_3}{\href{\githubsearchurl?q=\%22def+kzh_5_slope_q22_f10_3(\%22}{\sage{kzh\underscore{}5\underscore{}slope\underscore{}q22\underscore{}f10\underscore{}3}}}
\pgfkeyssetvalue{/sagefunc/kzh_5_slope_q22_f10_4}{\href{\githubsearchurl?q=\%22def+kzh_5_slope_q22_f10_4(\%22}{\sage{kzh\underscore{}5\underscore{}slope\underscore{}q22\underscore{}f10\underscore{}4}}}
\pgfkeyssetvalue{/sagefunc/kzh_5_slope_q22_f2_1}{\href{\githubsearchurl?q=\%22def+kzh_5_slope_q22_f2_1(\%22}{\sage{kzh\underscore{}5\underscore{}slope\underscore{}q22\underscore{}f2\underscore{}1}}}
\pgfkeyssetvalue{/sagefunc/kzh_6_slope_1}{\href{\githubsearchurl?q=\%22def+kzh_6_slope_1(\%22}{\sage{kzh\underscore{}6\underscore{}slope\underscore{}1}}}
\pgfkeyssetvalue{/sagefunc/kzh_6_slope_fulldim_covers_1}{\href{\githubsearchurl?q=\%22def+kzh_6_slope_fulldim_covers_1(\%22}{\sage{kzh\underscore{}6\underscore{}slope\underscore{}fulldim\underscore{}covers\underscore{}1}}}
\pgfkeyssetvalue{/sagefunc/kzh_6_slope_fulldim_covers_2}{\href{\githubsearchurl?q=\%22def+kzh_6_slope_fulldim_covers_2(\%22}{\sage{kzh\underscore{}6\underscore{}slope\underscore{}fulldim\underscore{}covers\underscore{}2}}}
\pgfkeyssetvalue{/sagefunc/kzh_6_slope_fulldim_covers_3}{\href{\githubsearchurl?q=\%22def+kzh_6_slope_fulldim_covers_3(\%22}{\sage{kzh\underscore{}6\underscore{}slope\underscore{}fulldim\underscore{}covers\underscore{}3}}}
\pgfkeyssetvalue{/sagefunc/kzh_6_slope_fulldim_covers_4}{\href{\githubsearchurl?q=\%22def+kzh_6_slope_fulldim_covers_4(\%22}{\sage{kzh\underscore{}6\underscore{}slope\underscore{}fulldim\underscore{}covers\underscore{}4}}}
\pgfkeyssetvalue{/sagefunc/kzh_6_slope_fulldim_covers_5}{\href{\githubsearchurl?q=\%22def+kzh_6_slope_fulldim_covers_5(\%22}{\sage{kzh\underscore{}6\underscore{}slope\underscore{}fulldim\underscore{}covers\underscore{}5}}}
\pgfkeyssetvalue{/sagefunc/kzh_7_slope_1}{\href{\githubsearchurl?q=\%22def+kzh_7_slope_1(\%22}{\sage{kzh\underscore{}7\underscore{}slope\underscore{}1}}}
\pgfkeyssetvalue{/sagefunc/kzh_7_slope_2}{\href{\githubsearchurl?q=\%22def+kzh_7_slope_2(\%22}{\sage{kzh\underscore{}7\underscore{}slope\underscore{}2}}}
\pgfkeyssetvalue{/sagefunc/kzh_7_slope_3}{\href{\githubsearchurl?q=\%22def+kzh_7_slope_3(\%22}{\sage{kzh\underscore{}7\underscore{}slope\underscore{}3}}}
\pgfkeyssetvalue{/sagefunc/kzh_7_slope_4}{\href{\githubsearchurl?q=\%22def+kzh_7_slope_4(\%22}{\sage{kzh\underscore{}7\underscore{}slope\underscore{}4}}}
\pgfkeyssetvalue{/sagefunc/ll_strong_fractional}{\href{\githubsearchurl?q=\%22def+ll_strong_fractional(\%22}{\sage{ll\underscore{}strong\underscore{}fractional}}}
\pgfkeyssetvalue{/sagefunc/mlr_cpl3_a_2_slope}{\href{\githubsearchurl?q=\%22def+mlr_cpl3_a_2_slope(\%22}{\sage{mlr\underscore{}cpl3\underscore{}a\underscore{}2\underscore{}slope}}}
\pgfkeyssetvalue{/sagefunc/mlr_cpl3_b_3_slope}{\href{\githubsearchurl?q=\%22def+mlr_cpl3_b_3_slope(\%22}{\sage{mlr\underscore{}cpl3\underscore{}b\underscore{}3\underscore{}slope}}}
\pgfkeyssetvalue{/sagefunc/mlr_cpl3_c_3_slope}{\href{\githubsearchurl?q=\%22def+mlr_cpl3_c_3_slope(\%22}{\sage{mlr\underscore{}cpl3\underscore{}c\underscore{}3\underscore{}slope}}}
\pgfkeyssetvalue{/sagefunc/mlr_cpl3_d_3_slope}{\href{\githubsearchurl?q=\%22def+mlr_cpl3_d_3_slope(\%22}{\sage{mlr\underscore{}cpl3\underscore{}d\underscore{}3\underscore{}slope}}}
\pgfkeyssetvalue{/sagefunc/mlr_cpl3_f_2_or_3_slope}{\href{\githubsearchurl?q=\%22def+mlr_cpl3_f_2_or_3_slope(\%22}{\sage{mlr\underscore{}cpl3\underscore{}f\underscore{}2\underscore{}or\underscore{}3\underscore{}slope}}}
\pgfkeyssetvalue{/sagefunc/mlr_cpl3_g_3_slope}{\href{\githubsearchurl?q=\%22def+mlr_cpl3_g_3_slope(\%22}{\sage{mlr\underscore{}cpl3\underscore{}g\underscore{}3\underscore{}slope}}}
\pgfkeyssetvalue{/sagefunc/mlr_cpl3_h_2_slope}{\href{\githubsearchurl?q=\%22def+mlr_cpl3_h_2_slope(\%22}{\sage{mlr\underscore{}cpl3\underscore{}h\underscore{}2\underscore{}slope}}}
\pgfkeyssetvalue{/sagefunc/mlr_cpl3_k_2_slope}{\href{\githubsearchurl?q=\%22def+mlr_cpl3_k_2_slope(\%22}{\sage{mlr\underscore{}cpl3\underscore{}k\underscore{}2\underscore{}slope}}}
\pgfkeyssetvalue{/sagefunc/mlr_cpl3_l_2_slope}{\href{\githubsearchurl?q=\%22def+mlr_cpl3_l_2_slope(\%22}{\sage{mlr\underscore{}cpl3\underscore{}l\underscore{}2\underscore{}slope}}}
\pgfkeyssetvalue{/sagefunc/mlr_cpl3_n_3_slope}{\href{\githubsearchurl?q=\%22def+mlr_cpl3_n_3_slope(\%22}{\sage{mlr\underscore{}cpl3\underscore{}n\underscore{}3\underscore{}slope}}}
\pgfkeyssetvalue{/sagefunc/mlr_cpl3_o_2_slope}{\href{\githubsearchurl?q=\%22def+mlr_cpl3_o_2_slope(\%22}{\sage{mlr\underscore{}cpl3\underscore{}o\underscore{}2\underscore{}slope}}}
\pgfkeyssetvalue{/sagefunc/mlr_cpl3_p_2_slope}{\href{\githubsearchurl?q=\%22def+mlr_cpl3_p_2_slope(\%22}{\sage{mlr\underscore{}cpl3\underscore{}p\underscore{}2\underscore{}slope}}}
\pgfkeyssetvalue{/sagefunc/mlr_cpl3_q_2_slope}{\href{\githubsearchurl?q=\%22def+mlr_cpl3_q_2_slope(\%22}{\sage{mlr\underscore{}cpl3\underscore{}q\underscore{}2\underscore{}slope}}}
\pgfkeyssetvalue{/sagefunc/mlr_cpl3_r_2_slope}{\href{\githubsearchurl?q=\%22def+mlr_cpl3_r_2_slope(\%22}{\sage{mlr\underscore{}cpl3\underscore{}r\underscore{}2\underscore{}slope}}}
\pgfkeyssetvalue{/sagefunc/rlm_dpl1_extreme_3a}{\href{\githubsearchurl?q=\%22def+rlm_dpl1_extreme_3a(\%22}{\sage{rlm\underscore{}dpl1\underscore{}extreme\underscore{}3a}}}
\pgfkeyssetvalue{/sagefunc/automorphism}{\href{\githubsearchurl?q=\%22def+automorphism(\%22}{\sage{automorphism}}}
\pgfkeyssetvalue{/sagefunc/multiplicative_homomorphism}{\href{\githubsearchurl?q=\%22def+multiplicative_homomorphism(\%22}{\sage{multiplicative\underscore{}homomorphism}}}
\pgfkeyssetvalue{/sagefunc/projected_sequential_merge}{\href{\githubsearchurl?q=\%22def+projected_sequential_merge(\%22}{\sage{projected\underscore{}sequential\underscore{}merge}}}
\pgfkeyssetvalue{/sagefunc/restrict_to_finite_group}{\href{\githubsearchurl?q=\%22def+restrict_to_finite_group(\%22}{\sage{restrict\underscore{}to\underscore{}finite\underscore{}group}}}
\pgfkeyssetvalue{/sagefunc/interpolate_to_infinite_group}{\href{\githubsearchurl?q=\%22def+interpolate_to_infinite_group(\%22}{\sage{interpolate\underscore{}to\underscore{}infinite\underscore{}group}}}
\pgfkeyssetvalue{/sagefunc/two_slope_fill_in}{\href{\githubsearchurl?q=\%22def+two_slope_fill_in(\%22}{\sage{two\underscore{}slope\underscore{}fill\underscore{}in}}}
\pgfkeyssetvalue{/sagefunc/generate_example_e_for_psi_n}{\href{\githubsearchurl?q=\%22def+generate_example_e_for_psi_n(\%22}{\sage{generate\underscore{}example\underscore{}e\underscore{}for\underscore{}psi\underscore{}n}}}
\pgfkeyssetvalue{/sagefunc/chen_3_slope_not_extreme}{\href{\githubsearchurl?q=\%22def+chen_3_slope_not_extreme(\%22}{\sage{chen\underscore{}3\underscore{}slope\underscore{}not\underscore{}extreme}}}
\pgfkeyssetvalue{/sagefunc/psi_n_in_bccz_counterexample_construction}{\href{\githubsearchurl?q=\%22def+psi_n_in_bccz_counterexample_construction(\%22}{\sage{psi\underscore{}n\underscore{}in\underscore{}bccz\underscore{}counterexample\underscore{}construction}}}
\pgfkeyssetvalue{/sagefunc/gomory_fractional}{\href{\githubsearchurl?q=\%22def+gomory_fractional(\%22}{\sage{gomory\underscore{}fractional}}}
\pgfkeyssetvalue{/sagefunc/not_minimal_2}{\href{\githubsearchurl?q=\%22def+not_minimal_2(\%22}{\sage{not\underscore{}minimal\underscore{}2}}}
\pgfkeyssetvalue{/sagefunc/not_extreme_1}{\href{\githubsearchurl?q=\%22def+not_extreme_1(\%22}{\sage{not\underscore{}extreme\underscore{}1}}}
\pgfkeyssetvalue{/sagefunc/kzh_2q_example_1}{\href{\githubsearchurl?q=\%22def+kzh_2q_example_1(\%22}{\sage{kzh\underscore{}2q\underscore{}example\underscore{}1}}}
\pgfkeyssetvalue{/sagefunc/zhou_two_sided_discontinuous_cannot_assume_any_continuity}{\href{\githubsearchurl?q=\%22def+zhou_two_sided_discontinuous_cannot_assume_any_continuity(\%22}{\sage{zhou\underscore{}two\underscore{}sided\underscore{}discontinuous\underscore{}cannot\underscore{}assume\underscore{}any\underscore{}continuity}}}
\pgfkeyssetvalue{/sagefunc/extremality_test}{\href{\githubsearchurl?q=\%22def+extremality_test(\%22}{\sage{extremality\underscore{}test}}}
\pgfkeyssetvalue{/sagefunc/plot_2d_diagram}{\href{\githubsearchurl?q=\%22def+plot_2d_diagram(\%22}{\sage{plot\underscore{}2d\underscore{}diagram}}}
\pgfkeyssetvalue{/sagefunc/nice_field_values}{\href{\githubsearchurl?q=\%22def+nice_field_values(\%22}{\sage{nice\underscore{}field\underscore{}values}}}
\pgfkeyssetvalue{/sagefunc/ParametricRealFieldElement}{\href{\githubsearchurl?q=\%22def+ParametricRealFieldElement(\%22}{\sage{ParametricRealFieldElement}}}
\pgfkeyssetvalue{/sagefunc/ParametricRealField}{\href{\githubsearchurl?q=\%22def+ParametricRealField(\%22}{\sage{ParametricRealField}}}

\DeclareRobustCommand\sage[1]{\texttt{#1}}
\DeclareRobustCommand\sagefunc[1]{\pgfkeys{/sagefunc/#1}}

\titlerunning{Facets, weak facets, and extreme functions}
\title{On the notions of\\ facets, weak facets, and extreme functions\\
  of the Gomory--Johnson infinite group problem%
  \thanks{The authors gratefully acknowledge partial support from the National Science
  Foundation through grant DMS-1320051, awarded to M.~K\"oppe.}}

\author{Matthias K\"oppe \and
  Yuan Zhou}

\institute{Dept.\ of Mathematics, University of California, Davis\\
  \texttt{mkoeppe@math.ucdavis.edu}, \texttt{yzh@math.ucdavis.edu}}

\date{$\relax$Revision: 2181 $ - \ $Date: 2016-11-20 16:21:21 -0800 (Sun, 20 Nov 2016) $ $\!\!\!}

\usepackage[normalem]{ulem}

\newcommand\Figure[2][\relax]{%
  \begin{figure}[h!]
    \includegraphics[width=.8\textwidth]{#2}
    \caption{\ifx#1\relax#2\else#1\fi}
  \end{figure}
}

\begin{document}
 \newcommand{\tgreen}[1]{\textsf{\textcolor {ForestGreen} {#1}}}
 \newcommand{\tred}[1]{\texttt{\textcolor {red} {#1}}}
 \newcommand{\tblue}[1]{\textcolor {blue} {#1}}

\maketitle

\begin{abstract}
  We investigate three competing notions that generalize the notion of a facet
  of finite-dimensional polyhedra to the infinite-dimension\-al Gomory--Johnson
  model.  These notions were known to coincide for continuous piecewise linear
  functions with rational breakpoints.  We show that two of the notions, extreme functions and
  facets, coincide for the case of continuous piecewise linear functions,
  removing the hypothesis regarding rational breakpoints.
  We then separate the three notions using discontinuous examples.
\end{abstract}

\section{Introduction}

\subsection{Facets in the finite-dimensional case}
Let $G$ be a finite index set.  The space~$\R^{(G)}$ of real-valued functions
$y\colon G\to\R$ is isomorphic to and routinely identified with the Euclidean
space $\R^{|G|}$.  Let $\R^{G}$ denote its dual space.  It is the space of
functions $\alpha\colon G \to \R$, which we consider as linear functionals on
$\R^{(G)}$ via the pairing $\langle \alpha, y \rangle = \sum_{r\in G} \alpha(r)
y(r)$.  Again it is routinely identified with the Euclidean space $\R^{|G|}$, 
and the dual pairing $\langle \alpha, y\rangle $ is the Euclidean inner product.
A~(closed, convex) rational polyhedron 
of~$\R^{(G)}$ is the set of $y\colon
G\to\R$ satisfying $\langle\alpha_i, y\rangle \geq \alpha_{i, 0}$, where
$\alpha_i\in \Z^G$ are integer linear functionals and $\alpha_{i,
  0}\in\Z$, for $i$ ranging over another finite index set~$I$. 

Consider an integer linear optimization problem in $\R^{(G)}$, i.e.,
the problem of minimizing a linear functional $\eta \in \R^G$ over 
a feasible set $F 
\subseteq \{\, y\colon G \to \Z_+\,\}
\subset \R_+^{(G)}$
, or, equivalently, 
over the convex hull $R = \conv F \subset \R_+^{(G)}$.  
A \emph{valid inequality} for $R$ is an inequality of the form
$\langle \pi, y \rangle \geq \pi_0$, where $\pi \in \R^G$, which holds for
all $y \in R$ (equivalently, for all $y \in F$).  If $R$ is closed and
convex, it is exactly the set of all $y$ that satisfy all valid inequalities. 
In the following we will restrict ourselves to the case that $R \subseteq
\R_+^{(G)}$ is a polyhedron of blocking type, in which case we only need to
consider normalized valid inequalities with $\pi\geq0$ and $\pi_0=1$. 

Let $P(\pi)$ denote the set of functions $y \in F$ for which the inequality
$\langle \pi, y \rangle \geq 1$ is tight, i.e.,
$\langle \pi, y \rangle = 1$.  If $P(\pi) \neq \emptyset$, then
$\langle \pi, y \rangle \geq 1$ is a \emph{tight valid inequality}.  Then
$R$ is exactly the set of all $y\geq0$ that satisfy all tight valid inequalities.
A valid inequality $\langle \pi, y \rangle \geq 1$ is called
\emph{minimal} 
if there is no other valid inequality $\pi' \neq \pi$ such that $\pi' \leq \pi$ pointwise.
One can show that a minimal valid inequality is tight.
A valid inequality $\langle \pi, y \rangle \geq 1$ is called
\emph{facet-defining} if 
\begin{equation}\tag{wF}
  \begin{aligned}
    \text{for every valid inequality $\langle \pi', y \rangle \geq 1$
      such that $P(\pi)\subseteq P(\pi')$},\\
    \text{we have $P(\pi)=P(\pi')$,}
  \end{aligned}
  \label{eq:facet-definition-like-weak-facet}
\end{equation}
or, in other words, if the face induced by $\langle \pi, y \rangle \geq 1$
is maximal. 
Under the above assumptions, $R$ has full affine dimension.  Thus, we get the
following characterization of facet-defining inequalities:
\begin{equation}\tag{F}
  \begin{aligned}
    \text{for every valid inequality $\langle \pi', y \rangle \geq 1$
      such that $P(\pi)\subseteq P(\pi')$},\\
    \text{we have $\pi = \pi'$.}
  \end{aligned}
  \label{eq:facet-definition-like-facet}
\end{equation}
The theory of polyhedra gives another characterization of facets:
\begin{equation}\tag{E}
   \begin{aligned}
     \text{If $\langle \pi^1, y \rangle \geq 1$ and $\langle \pi^2, y \rangle
       \geq 1$ are valid inequalities, and $\pi = \tfrac12 (\pi^1+\pi^2)$}\\
     \text{then $\pi = \pi^1 = \pi^2$}. 
     \label{eq:facet-definition-like-extreme-function}
   \end{aligned}
\end{equation}




\subsection{Facets in the infinite-dimensional Gomory--Johnson model}



It is perhaps not surprising that the three conditions
\eqref{eq:facet-definition-like-weak-facet},
\eqref{eq:facet-definition-like-facet}, and
\eqref{eq:facet-definition-like-extreme-function} are no longer equivalent
when $R$ is a general convex set that is not polyhedral, and in particular when we change from the
finite-dimensional to the infinite-dimensional case. 
In the present paper, however, we consider a particular case of an
infinite-dimensional model, in which this question has eluded researchers for
a long time. Let $G=\Q$ or $G=\R$ and let $\R^{(G)}$ now denote the space of
finite-support functions $y\colon G\to\R$. 
The so-called \emph{infinite group problem} was introduced
by Gomory and Johnson in their seminal papers \cite{infinite,infinite2}. 
Let $F = F_{f}(G,\Z)\subseteq \R_+^{(G)}$ be the set of all finite-support
functions $y\colon G\to\Z_+ $ satisfying the equation 
\begin{equation}
  \label{GP} 
  \sum_{r \in G} r\, y(r) \equiv f \pmod{1}
\end{equation}
where $f$ is a given element of $G\setminus \Z$. 
We study its convex hull $R = R_{f}(G,\Z) \subseteq \R_+^{(G)}$, whose elements are understood as
finite-support functions $y\colon G \to \R_+$. 

Valid inequalities for $R$ are of the form
$\langle \pi, y \rangle \geq \pi_0$, where $\pi$ comes from the dual space
$\R^G$, which is the space of all real-valued functions (without the
finite-support condition).  When $G=\Q$, then $R$ is again of ``blocking
type'' (see, for example, \cite[section 5]{corner_survey}), and so we again
may assume $\pi\geq0$ and $\pi_0=1$.

If $G=\R$ (the setting of the present paper), typical pathologies from the
analysis of functions of a real 
variable come into play.  For example, by \cite[Proposition 2.4]{igp_survey}
there is an infinite-dimensional space of valid equations $\langle\pi^*,
y\rangle=0$, where $\pi^*$ are constructed using a Hamel basis of $\R$
over~$\Q$. Each of these functions~$\pi^*$ has a graph whose topological
closure is~$\R^2$.
In order to tame these pathologies, it is common to make further assumptions.
Gomory--Johnson \cite{infinite,infinite2} only considered continuous
functions~$\pi$.  However, this rules out many interesting functions such as
the Gomory fractional cut. Instead it has become common in the literature to build the assumption
$\pi\geq0$ into the definition; then we can again normalize $\pi_0=1$.  We
call such functions $\pi$ \emph{valid functions}.

(Minimal) valid functions~$\pi$ that satisfy the conditions
\eqref{eq:facet-definition-like-weak-facet},
\eqref{eq:facet-definition-like-facet}, and
\eqref{eq:facet-definition-like-extreme-function}, are called \emph{weak
  facets}, \emph{facets}, and \emph{extreme functions}, respectively. 
The relation of these notions, in particular of facets and extreme functions, 
has remained unclear in the literature.  For example, Basu et
al. \cite{bccz08222222} wrote: 
\begin{quote}
  The statement that extreme functions are facets appears to be quite
  nontrivial to prove, and to the best of our knowledge there is no proof in
  the literature. We therefore cautiously treat extreme functions and facets
  as distinct concepts, and leave their equivalence as an open question.
\end{quote}
We refer to \cite[section 2.2]{igp_survey} for further discussion.


\subsection{Contribution of this paper}

A well known sufficient condition for facetness of a minimal valid function~$\pi$ is
the Gomory--Johnson Facet Theorem.  In its strong form, due to
Basu--Hildebrand--K\"{o}ppe--Molinaro \cite{basu-hildebrand-koeppe-molinaro:k+1-slope}, it reads:
\begin{theorem}[{Facet Theorem, strong form, \cite[Lemma
34]{basu-hildebrand-koeppe-molinaro:k+1-slope}; see also \cite[Theorem
2.12]{igp_survey}}]
\label{thm:facet-theorem-strong-form}
  Suppose for every minimal valid function $\pi'$, $E(\pi) \subseteq E(\pi')$
  implies $\pi' = \pi$.  Then $\pi$ is a facet.
\end{theorem}
(Here $E(\pi)$ is the \emph{additivity domain} of $\pi$, defined in
\autoref{sec:notations}.) 
We show (\autoref{lemma:facet_theorem} below) that, in fact, \textbf{this holds as an
``if and only if'' statement.} 

For the case of continuous piecewise linear functions with rational
breakpoints, Basu et al. \cite[Proposition 2.8]{igp_survey} showed that the
notions of extreme functions and facets coincide.  This was a consequence of
Basu et al.'s finite oversampling theorem \cite{basu-hildebrand-koeppe:equivariant}. 
We \textbf{sharpen this result by removing the hypothesis regarding rational
breakpoints.}

\begin{theorem} 
\label{lemma:cont_pwl_extreme_is_facet}
In the case of continuous piecewise linear functions (not necessarily with rational breakpoints), 
$\{$extreme functions$\}=\{$facets$\}$. 
\end{theorem}

Then we investigate the notions of facets and weak
facets in the case of discontinuous functions.  This appears to be a first in
the published literature.  All papers that consider discontinuous functions
only used the notion of extreme functions; see Appendix~\ref{appendix:discont-lit}
for a discussion.
We give \textbf{three discontinuous functions that furnish the separation of the three
notions} (\autoref{lemma:discontinuous_examples}): 
A function $\psi$ that is extreme, but is neither a weak facet nor a facet;
a function $\pi$ that is not an extreme function (nor a facet), but is a weak
facet;
and a function $\pi_{\mathrm{lifted}}$ that is extreme and a weak facet but is
not a facet.  

It remains an open question whether this separation can also be done using
continuous (non--piecewise linear) functions.

\section{Minimal valid functions and their perturbations}
\label{sec:notations}
Following \cite{igp_survey}, given a locally finite one-dimensional polyhedral complex $\P$, we call a function $\pi\colon \R \to \R$ \emph{piecewise linear} over $\P$, if it is affine linear over the relative interior of each face of the complex. Under this definition, piecewise linear functions can be discontinuous. We say the function $\pi$ is \emph{continuous piecewise linear} over $\P$ if it is affine over each of the cells of $\P$ (thus automatically imposing continuity).

Let $\pi$ be a minimal valid function.  
Define the subadditivity slack of $\pi$ as $\Delta\pi(x,y) := \pi(x) + \pi(y)
- \pi(x+y)$. Denote the \emph{additivity domain} of $\pi$ by
$$E(\pi) = \{(x,y) \st \Delta\pi(x,y) = 0\}.$$

To combinatorialize the additivity domains of piecewise linear subadditive
functions, we work with the two-dimensional polyhedral complex $\Delta \P 
$, 
whose faces are $F(I, J, K) = \{(x,y) \in \R \times \R \mid x \in I, y \in J,
x+y \in K\}$
for $I, J, K \in \P$.  
Define the projections $p_1,p_2,p_3\colon \R\times \R \to \R$ as
$p_1(x,y) = x$,  $p_2(x,y) = y$, $p_3(x,y) = x+y$.

In the continuous case, since the function $\pi$ is piecewise linear over
$\P$, we have that $\Delta \pi$ is affine linear over each face
$F \in \Delta P$.  
Let $\pi$ be a  minimal valid function for $R_f(\R, \Z)$ that is piecewise linear over $\P$. 
Following \cite{igp_survey}, we define the \emph{space of perturbation functions with prescribed additivities} $E = E(\pi)$ 
\begin{equation}
\bar{\Pi}^{E}(\R,\Z) = \left\{\bar{\pi} \colon \R \to \R \, \Bigg| \,
\begin{array}{r@{\;}c@{\;}ll}
\bar{\pi}(0) &=& 0 \\
\bar{\pi}(f) &=& 0 \\
\bar{\pi}(x) + \bar{\pi}(y) &=& \bar{\pi}(x+y) & \text{ for all } (x,y) \in E\\
\bar{\pi}(x) &=& \bar{\pi}(x+t) & \text{ for all } x \in \R,\, t \in \Z
\end{array} \right\}.
\label{eq:perturbation_space_simple}
\end{equation}

When $\pi$ is discontinuous, one also needs to consider the limit points where the subadditivity slacks are approaching zero. Let $F$ be a face of $\Delta \P $. For $(x,y)\in F$, we denote 
\[\Delta\pi_F(x,y) := 
\lim_{\substack{(u,v) \to (x,y)\\ (u,v) \in \relint(F)}} \Delta\pi(u,v).\]
Define 
\[E_F(\pi) = \{\,(x,y)\in F \st \Delta\pi_F(x,y) \text{ exists, and } \Delta\pi_F(x,y) = 0\,\}.\]
Notice that in the above definition of $E_F(\pi)$, we include the condition that
the limit denoted by $\Delta\pi_F(x,y)$ exists, so that this definition can as well be applied to
functions $\pi$ (and $\bar\pi$) that are not piecewise linear over $\P$. 

We denote by $E_{\bullet}(\pi, \P)$ the family of sets $E_F(\pi)$, indexed by $F\in \Delta\P$.
Define the \emph{space of perturbation functions with prescribed additivities and limit-additivities} $E_{\bullet} = E_{\bullet}(\pi, \P)$
\begin{equation}
\bar{\Pi}^{E_{\bullet}}(\R,\Z) = \left\{\bar{\pi} \colon \R \to \R \, \Bigg| \,
\begin{array}{r@{\;}c@{\;}ll}
\bar{\pi}(0) &=& 0 \\
\bar{\pi}(f) &=& 0 \\
\Delta\bar{\pi}_F(x, y) &=& 0 & \text{ for } (x,y) \in E_F, \ F \in \Delta\P\\
\bar{\pi}(x) &=& \bar{\pi}(x+t) & \text{ for } x \in \R,\, t \in \Z
\end{array} \right\}.
\label{eq:perturbation_space}
\end{equation}
\begin{remark}
Let $\bar{\pi} \in \bar{\Pi}^E(\R,\Z)$. The third condition of \eqref{eq:perturbation_space_simple} is equivalent to $E(\pi) \subseteq E(\bar{\pi})$. 
Let $\bar{\pi} \in \bar{\Pi}^{E_{\bullet}}(\R,\Z)$. The third condition of \eqref{eq:perturbation_space} is equivalent to $E_F(\pi) \subseteq E_F(\bar{\pi})$ for all faces $F\in \Delta\P$, which is stronger than $E(\pi) \subseteq E(\bar{\pi})$ in \eqref{eq:perturbation_space_simple}.
Thus, in general, $\bar{\Pi}^{E_{\bullet}}(\R,\Z) \subseteq \bar{\Pi}^E(\R,\Z)$.  
If $\pi$ is continuous, then $E(\pi) \subseteq E(\bar{\pi})$ implies that $E_F(\pi) \subseteq E_F(\bar{\pi})$ for all faces $F\in \Delta\P$, hence $\bar{\Pi}^{E_{\bullet}}(\R,\Z) = \bar{\Pi}^E(\R,\Z)$.  
\end{remark}

\section{Effective perturbation functions}

Following \cite{koeppe-zhou:crazy-perturbation}, 
we define the \emph{space of effective perturbation functions}
\begin{equation}
\tilde{\Pi}^{\pi}(\R,\Z) = \left\{\,\tilde{\pi} \colon \R \to \R \, \mid \, \exists \, \epsilon>0 \text{ s.t.\ } \pi^{\pm} = \pi \pm \epsilon\tilde{\pi} \text{ are minimal valid}\,\right\}.
\label{eq:effective-perturbation-space}
\end{equation}
Because of \cite[Lemma 2.11\,(i)]{igp_survey}, a function $\pi$ is
extreme if and only if $\tilde{\Pi}^{\pi}(\R,\Z) = \left\{ 0\right\}$. Note
that every
function $\tilde{\pi} \in \tilde{\Pi}^{\pi}(\R,\Z)$ is bounded
.

It is clear that
if $\tilde\pi \in \tilde{\Pi}^{\pi}(\R,\Z)$, then $\tilde\pi \in
\bar{\Pi}^{E_{\bullet}}(\R,\Z)$, where $E_{\bullet} = E_{\bullet}(\pi, \P)$; see
\cite[Lemma 2.7]{basu-hildebrand-koeppe:equivariant} or
\cite[{Lemma~\ref{crazy:lemma:tight-implies-tight}}]{koeppe-zhou:crazy-perturbation}.

The other direction does not hold in general, but requires additional
hypotheses.  Let $\bar\pi \in \bar{\Pi}^{E_{\bullet}}(\R,\Z)$.  In
\cite[Theorem 3.13]{bhk-IPCOext} (see also \cite[Theorem 3.13]{igp_survey}),
it is proved that if $\pi$ and $\bar\pi$ are continuous
and $\bar\pi$ is piecewise linear, we have
$\bar\pi \in \tilde{\Pi}^{\pi}(\R,\Z)$.  (Similar arguments also appeared in
the earlier literature, for example in the proof of \cite[Theorem
3.2]{basu-hildebrand-koeppe:equivariant}.)

We will need a more general version of this result. 
Consider the following definition.
Given a locally finite one-dimensional polyhedral complex $\mathcal{P}$, 
we call a function $\bar\pi\colon \R \to \R$ \textit{piecewise Lipschitz continuous} over~$\mathcal{P}$, if it is Lipschitz continuous over the relative interior of each face of the complex. Under this definition, piecewise Lipschitz continuous functions can be discontinuous.
\begin{theorem}
\label{lemma:lipschitz-equiv-perturbation}
Let $\pi$ be a minimal valid function that is piecewise linear over a
polyhedral complex $\P$. Let $\bar\pi \in \bar\Pi^{E_{\bullet}}(\R,\Z)$ be a
perturbation function, where $E_{\bullet} = E_{\bullet}(\pi, \P)$. Suppose
that $\bar\pi$ is piecewise Lipschitz continuous over $\P$. Then $\bar\pi$ is
an effective perturbation function, $\bar\pi \in
\tilde \Pi^{\pi}(\R,\Z)$.  
\end{theorem} 
The proof appears in Appendix~\ref{appendix:omitted}.


\section{Extreme functions and facets}
In this section, we discuss the relations between the notions of extreme
functions and facets.
We first review the definition of a facet, following
\cite[section 2.2.3]{igp_survey}; cf.\ ibid.\ for a discussion of this notion
in the earlier literature, in particular 
\cite{tspace} 
and \cite{dey3}
.

Let $P(\pi)$ denote the set of functions $y\colon \R \to \Z_+$ with finite
support satisfying \[\sum_{r\in\R}r y(r) \in f + \Z \quad \text{ and } \quad
  \sum_{r\in\R}\pi(r)y(r)=1.\] A valid function $\pi$ is called a \emph{facet}
if for every valid function $\pi'$ such that $P(\pi) \subseteq P(\pi')$ we
have that $\pi' =\pi$
. Equivalently, a valid function $\pi$ is a facet if this condition holds
for all such \emph{minimal} valid functions $\pi'$ \cite{basu-hildebrand-koeppe-molinaro:k+1-slope}. 
\begin{remark}
In the discontinuous case, the additivity in the limit plays a role in
extreme functions, which are characterized by the non-existence of an
effective perturbation function $\tilde{\pi}\not\equiv 0$. However facets (and
weak facets, see the next section) are defined through $P(\pi)$, which does not capture the limiting
additive behavior of $\pi$. The additivity domain $E(\pi)$, which features in
the Facet Theorem as discussed below, also does not account for additivity in
the limit.
\end{remark}


A well known sufficient condition for facetness of a minimal valid function~$\pi$ is
the Gomory--Johnson Facet Theorem.  We have stated its strong form, due to
Basu--Hildebrand--K\"{o}ppe--Molinaro
\cite{basu-hildebrand-koeppe-molinaro:k+1-slope}, in the introduction as
\autoref{thm:facet-theorem-strong-form}. 
In order to prove our ``if and only if'' version, we need the following
lemma. 
\begin{lemma}
\label{lemma:P_pi_and_E_pi}
Let $\pi$ and $\pi'$ be minimal valid functions. Then
$E(\pi) \subseteq E(\pi')$ if and only if $P(\pi) \subseteq P(\pi')$. 
\end{lemma}
\begin{proof}
The ``if'' direction is proven in
\cite[Theorem 20]{basu-hildebrand-koeppe-molinaro:k+1-slope}; see also \cite[Theorem 2.12]{igp_survey}. 
We now show the ``only if" direction, using the subadditivity of $\pi$. Assume that $E(\pi) \subseteq E(\pi')$. Let $y \in P(\pi)$. Let $\{r_1, r_2, \dots, r_n\}$ denote the finite support of $y$. By definition, the function $y$ satisfies that $y(r_i) \in \Z_+$,  $\sum_{i=1}^n r_i y(r_i) \equiv f \pmod 1$, and $\sum_{i=1}^n \pi(r_i) y(r_i) = 1$. 
Since $\pi$ is a minimal valid function, we have that 
\(1 = \sum_{i=1}^n \pi(r_i) y(r_i) \geq \pi(\sum_{i=1}^n r_i y(r_i)) = \pi(f) = 1.\) 
Thus, each subadditivity inequality here is tight for $\pi$, and is also tight for $\pi'$ since $E(\pi) \subseteq E(\pi')$.  We obtain \(\sum_{i=1}^n \pi'(r_i) y(r_i) = \pi'(\sum_{i=1}^n r_i y(r_i)) = \pi'(f) = 1,\) which implies that $y \in P(\pi')$. Therefore, $P(\pi) \subseteq P(\pi')$.
\end{proof}

\begin{theorem}[Facet Theorem, ``if and only if'' version]
\label{lemma:facet_theorem}
A minimal valid function $\pi$ is a facet if and only if for every minimal valid function $\pi'$, $E(\pi) \subseteq E(\pi')$ implies $\pi'=\pi$. 
\end{theorem}
\begin{proof}
It follows from the Facet Theorem in the strong form
(\autoref{thm:facet-theorem-strong-form}) and \autoref{lemma:P_pi_and_E_pi}.
\end{proof}

In \cite[page 25, section 3.6]{igp_survey}, the Facet Theorem is reformulated in terms of
perturbation functions.  In Appendix~\ref{appendix:facet-theorem-with-perturbation}
we prove an ``if and only if'' result for this reformulation as well. \medbreak

Now we come to the proof of a main theorem stated in the introduction.

\begin{proof}[of \autoref{lemma:cont_pwl_extreme_is_facet}]
Let $\pi$ be a continuous piecewise linear minimal valid function. As
mentioned in \cite[section 2.2.4]{igp_survey}, \cite[Lemma
1.3]{basu-hildebrand-koeppe-molinaro:k+1-slope}
showed that if $\pi$ is a facet, then $\pi$ is extreme.

We now prove the other direction by contradiction. Suppose that $\pi$ is extreme, but is not a facet. 
Then by \autoref{lemma:facet_theorem}, there exists a minimal valid function $\pi' \neq \pi$ such that $E(\pi)\subseteq E(\pi')$. 
Since $\pi$ is continuous piecewise linear and $\pi(0)=\pi(1)=0$, 
there exists $\delta >0$ such that $\Delta\pi(x,y) =0$ and $\Delta\pi(-x, -y) =0$ for $0 \leq x, y \leq \delta$. The condition $E(\pi) \subseteq E(\pi')$ implies that $\Delta\pi'(x,y) =0$ and $\Delta\pi'(-x, -y) =0$ for $0 \leq x, y \leq \delta$ as well. As the function $\pi'$ is bounded, it follows from the Interval Lemma (see \cite[Lemma 4.1]{igp_survey}, for example) that $\pi'$ is affine linear on $[0, \delta]$ and on $[-\delta, 0]$. We also know that $\pi'(0)=0$ as $\pi'$ is minimal valid. Using the subadditivity, we obtain that $\pi'$ is Lipschitz continuous.  Let $\bar\pi = \pi' -\pi$. Then $\bar\pi \not\equiv 0$, $\bar\pi \in \bar\Pi^E(\R, \Z)$ where $E = E(\pi)$, and $\bar\pi$ is Lipschitz continuous. Since $\pi$ is continuous, we have $\bar\Pi^E(\R, \Z) = \bar\Pi^{E_{\bullet}}(\R,\Z)$. By \autoref{lemma:lipschitz-equiv-perturbation}, there exists $\epsilon>0$ such that $\pi^\pm = \pi \pm \epsilon\bar\pi$ are distinct minimal valid functions. This contradicts the assumption that $\pi$ is an extreme function. 

Therefore, $\{$extreme functions$\} = \{$facets$\}$.
\end{proof}



\section{Weak facets}

We first review the definition of a weak facet, following
\cite[section 2.2.3]{igp_survey}; cf.\ ibid.\ for a discussion of this notion
in the earlier literature, in particular 
\cite{tspace} 
and \cite{dey3}
.
A valid function $\pi$ is called a \emph{weak facet} if for every valid
function $\pi'$ such that $P(\pi)\subseteq P(\pi')$ we have $P(\pi)=P(\pi')$.

As we mentioned above, to prove that $\pi$ is an extreme function or is a
facet, it suffices to consider $\pi'$ that is minimal valid. The following
lemma shows it is also the case in the definition of weak facets. 

\begin{lemma}
\label{lemma:weak_facet_minimal}
\begin{enumerate}[\rm(1)]
\item Let $\pi$ be a valid function. If $\pi$ is a weak facet, then $\pi$ is minimal valid.
\item Let $\pi$ be a minimal valid function. Suppose that for every minimal valid function $\pi'$, we have that $P(\pi)\subseteq P(\pi')$ implies $P(\pi)=P(\pi')$. Then $\pi$ is a weak facet.
\item A minimal valid function $\pi$ is a weak facet if and only if for every
  minimal valid function $\pi'$, we have that $E(\pi)\subseteq E(\pi')$
  implies $E(\pi)=E(\pi')$.
\end{enumerate}
\end{lemma}
\begin{proof}
(1) Suppose that $\pi$ is not minimal valid. Then, by \cite[Theorem
1]{basu-hildebrand-koeppe-molinaro:k+1-slope}, $\pi$ is dominated by another
minimal valid function $\pi'$, with $\pi(x_0) > \pi'(x_0)$ at some $x_0$. 
Let $y \in P(\pi)$. We have \[1 = \sum\pi(r_i)y(r_i) \geq \sum\pi'(r_i)y(r_i) \geq \pi'(\sum r_i y(r_i)) = \pi'(f) =1.\] Hence equality holds throughout, implying that $y \in P(\pi')$. Therefore, $P(\pi)\subseteq P(\pi')$. Now consider $y$ with $y(x_0)=y(f-x_0)=1$ and $y(x)=0$ otherwise. It is easy to see that $y \in P(\pi')$, but $y \not\in P(\pi)$ since $\pi(x_0)+\pi(f-x_0) > \pi'(x_0)+\pi'(f-x_0) = 1$. Therefore, $P(\pi)\subsetneq P(\pi')$, a contradiction to the weak facet assumption on $\pi$.\smallbreak

(2) Consider any valid function $\pi^*$ (not necessarily minimal) such that $P(\pi) \subseteq P(\pi^*)$. Let $\pi'$ be a minimal function that dominates $\pi^*$: $\pi' \leq \pi^*$. From (1) we know that $P(\pi^*) \subseteq P(\pi')$. Thus, $P(\pi)\subseteq P(\pi')$. By hypothesis,  we have that $P(\pi)=P(\pi^*)=P(\pi')$. Therefore, $\pi$ is a weak facet.\smallskip

(3) Direct consequence of (2) and \autoref{lemma:P_pi_and_E_pi}.
\end{proof}


\begin{theorem}\label{thm:implications-of-pwl-effective-perturbation}
Let $\mathcal{F}$ be a family of functions such that existence of an effective perturbation implies existence of a piecewise linear effective perturbation. Let $\pi$ be a continuous piecewise linear function (not necessarily with rational breakpoints) such that  $\pi \in \mathcal{F}$. The following are equivalent.
(1) $\pi$ is extreme, (2) $\pi$ is a facet, (3) $\pi$ is a weak facet.
\end{theorem}
\begin{remark}
  In general, facets form a subset of the intersection of extreme functions
  and weak facets.  In the case of continuous piecewise linear functions with
  rational breakpoints, \cite[Proposition 2.8]{igp_survey} and
  \cite[Theorem~8.6]{igp_survey_part_2} proved that ``extreme
  $\Leftrightarrow$ facet''. Note that in this case, ``weak facet
  $\Rightarrow$ facet'' can be shown by restriction with oversampling to
  finite group problems. Thus (1), (2), (3) are equivalent when $\pi$ is a
  continuous piecewise linear function with rational breakpoints. See
  \cite[Figure 2]{igp_survey} for an illustration. As shown in
  \cite{basu-hildebrand-koeppe:equivariant} (for a stronger statement, see
  \cite[Theorem~8.6]{igp_survey_part_2}), the family of continuous piecewise
  linear function with rational breakpoints is such a family $\mathcal{F}$
  where existence of an effective perturbation implies existence of a
  piecewise linear effective perturbation. A forthcoming paper will
  investigate larger such families $\mathcal{F}$.
\end{remark}
\begin{proof}[of \autoref{thm:implications-of-pwl-effective-perturbation}]
By \autoref{lemma:cont_pwl_extreme_is_facet} and the fact that $\{$facets$\} \subseteq \{$extreme functions$\} \cap \{$weak facets$\}$, it suffices to show that $\{$weak facets$\} \subseteq \{$extreme functions$\}$.

Assume that $\pi$ is a weak facet, thus $\pi$ is minimal valid by \autoref{lemma:weak_facet_minimal}. We show that $\pi$ is extreme. For the sake of contradiction, suppose that $\pi$ is not extreme. By the assumption $\pi \in \mathcal{F}$, there exists a piecewise linear perturbation function $\bar\pi \not\equiv 0$ such that $\pi\pm\bar\pi$ are minimal valid functions. Furthermore, by \cite[Lemma 2.11]{igp_survey},  
we know that $\bar\pi$ is continuous, and $E(\pi)\subseteq E(\bar\pi)$. 
By taking the union of the breakpoints, 
we can define a common refinement, which will still be denoted by $\P$, of the complexes for $\pi$ and for $\bar\pi$. In other words, we may assume that $\pi$ and $\bar\pi$ are both continuous piecewise linear over $\P$.  
Since $\Delta\bar\pi\not\equiv0$, we may assume without loss of generality that 
$\Delta\bar\pi(x,y) > 0$ for some $(x,y)\in \verts(\Delta\P)$.
Define \[\epsilon = \min\left\lbrace \frac{\Delta\pi(x,y)}{\Delta\bar\pi(x,y)}\st (x,y)\in \verts(\Delta\P),  \Delta\bar\pi(x,y) > 0\right\rbrace.\]
Notice that $\epsilon >0$, since $\Delta\pi \geq 0$ and $E(\pi)\subseteq E(\bar\pi)$.
Let $\pi'=\pi - \epsilon\bar\pi$. Then $\pi'$ is a bounded continuous function piecewise linear over $\P$, such that $\pi' \neq \pi$. 

The function $\pi'$ is subadditive, since $\Delta\pi'(x,y) \geq 0$ for each $(x,y)\in \verts(\Delta\P)$. As in the proof of \autoref{lemma:lipschitz-equiv-perturbation}, it can be shown that $\pi'$ is non-negative, $\pi'(0)=0$, $\pi'(f)=1$, and that $\pi'$ satisfies the symmetry condition. Therefore, $\pi'$ is a minimal valid function. Let $(u,v)$ be a vertex of $\Delta\P$ satisfying $\Delta\bar\pi(u,v) > 0$ and $\Delta\pi(u,v) = \epsilon\Delta\bar\pi(u,v)$. We know that $\Delta\pi'(u,v)=\Delta\pi(u,v)-\epsilon\Delta\bar\pi(u,v)=0$, hence $(u,v)\in E(\pi')$. However, $(u,v)\not\in E(\pi)$, since $\Delta\bar\pi(u,v)> 0$ implies that $\Delta\pi(u,v)\neq 0$.  Therefore, $E(\pi)\subsetneq E(\pi')$. By \autoref{lemma:weak_facet_minimal}(3), we have that $\pi$ is not a weak facet, a contradiction.
\end{proof}

\section{Separation of the notions in the discontinuous case}


The definition of facets fails to account for
additivities-in-the-limit, which are a crucial feature of the extremality test
for discontinuous functions.  
This allows us to separate the two
notions.  Below we do this by observing that a discontinuous piecewise linear
extreme function from the literature,
\sagefunc{hildebrand_discont_3_slope_1}(), constructed by Hildebrand
(2013, unpublished; reported in \cite{igp_survey}), works as a separating
example.  

The other separations appear to require more complicated
constructions.  Recently, the authors constructed a two-sided discontinuous
piecewise linear minimal valid function,
\sage{kzh\_minimal\_has\_only\_crazy\_perturbation\_1}, which is not extreme,
but which is not a convex combination of other piecewise linear minimal valid
functions; see \cite{koeppe-zhou:crazy-perturbation} for the definition.
This function has two special ``uncovered'' pieces on the intervals $(l, u)$
and $(f-u, f-l)$, where $f = \frac{4}{5}$, $l=\frac{219}{800}$,
$u=\frac{269}{800}$, on which every nonzero perturbation is microperiodic
(invariant under the action of the dense additive group $T = \langle t_1,
t_2 \rangle_{\Z}$, where $t_1 = \frac{77}{7752}\sqrt{2}$, $t_2 = \frac{77}{2584}$). 
Below we prove that it furnishes another separation.  

For the remaining separation, we construct an extreme function
$\pi_{\mathrm{lifted}}$ as follows.  Define $\pi_{\mathrm{lifted}}$ by perturbing
the function $\pi = \sage{kzh\_minimal\_has\_only\_}$\allowbreak$\sage{crazy\_perturbation\_1()}$ on infinitely
many cosets of the group~$T$ on the two uncovered intervals as follows.  
\begin{equation}\label{eq:lifted_crazy_example}
      \pi_{\mathrm{lifted}}(x)= 
      \begin{cases} 
        \pi(x) & \text{if } x \not\in (l, u) \cup (f-u, f-l) \text{, or} \\
        & \text{if } x \in (l, u) \text{ such that } x \in C \text{, or} \\
        & \text{if } x \in (f-u, f-l) \text{ such that } f - x \in  C; \\
        \pi(x) + s & \text{if } x \in (l, u) \text{ such that } x \in C^+  \text{, or} \\
        & \text{if } x \in (f-u, f-l) \text{ such that } f - x \in C^+; \\
        \pi(x) - s  & \text{otherwise},
      \end{cases}
    \end{equation}
    where $x_{39} = \frac{4899}{5000}$, 
    $s = \pi(x_{39}^-)+\pi(1+l-x_{39})-\pi(l) = \frac{19}{23998}$,
    \begin{align*}
      C &= \lbrace  x \in \R/T \st x = \tfrac{l+u}{2} \text{ or }
      \tfrac{l+u-t_1}{2} \text{ or }  \tfrac{l+u-t_2}{2}\rbrace, \\
      C^+ &=\lbrace x \in \R/T \st \text{arbitrary choice of one element of }\{x, \phi(x)\},\, x \not\in C \rbrace
    \end{align*}
    with $\phi\colon \R/T \ni x \mapsto l+u - x$. 

\begin{theorem}
\label{lemma:discontinuous_examples}
\begin{enumerate}[\rm(1)]
\item The function $\psi= \sage{hildebrand\_discont\_3\_slope\_1()}$ is extreme, but is neither a weak facet nor a facet.
\item The function $\pi= \sage{kzh\_minimal\_has\_only\_crazy\_perturbation\_1()}$ is not an extreme function (nor a facet), but is a weak facet.
\item The function $\pi_{\mathrm{lifted}}$ is extreme; it is a
  weak facet but is not a facet.
\end{enumerate}
\end{theorem} 

\begin{proof}
(1) The function $\psi= \sage{hildebrand\_discont\_3\_slope\_1()}$ is extreme
(Hildebrand, 2013, unpublished, reported in \cite{igp_survey}).  
This can be verified using the extremality test implemented in \cite{infinite-group-relaxation-code}. Consider the minimal valid function $\pi'$ defined by
\[
\pi'(x)= 
      \begin{cases} 
        2x & \text{if } x \in [0, \frac12]; \\
        \pi(x) & \text{if } x \in (\frac12, 1).
      \end{cases}
\]
Observe that $E(\psi)$ is a strict subset of $E(\pi')$. 
See \autoref{fig:simple_E_pi_extreme_not_facet} for an illustration. Thus, by \autoref{lemma:weak_facet_minimal}(3), the function $\psi$ is not a weak facet (nor a facet). 

\smallbreak

(2) By \cite[{Theorem~\ref{crazy:th:kzh_minimal_has_only_crazy_perturbation_1}}]{koeppe-zhou:crazy-perturbation}, the function $\pi= \sage{kzh\_minimal\_has\_only\_crazy\_}$\allowbreak$\sage{perturbation\_1()}$ is minimal valid, but is not extreme. Let $\pi'$ be a minimal valid function such that $E(\pi) \subseteq E(\pi')$. We want to show that $E(\pi)= E(\pi')$. Consider $\bar\pi = \pi' - \pi$, which is a bounded $\Z$-periodic function satisfying that $E(\pi) \subseteq E(\bar\pi)$. 
We apply the proof of \cite[{Theorem~\ref{crazy:th:kzh_minimal_has_only_crazy_perturbation_1}}, Part (ii)]{koeppe-zhou:crazy-perturbation} to the perturbation $\bar\pi$, and obtain that \begin{enumerate}
\item[(i)] $\bar\pi(x)=0$ for $x \not\in (l, u) \cup (f-u, f-l)$; 
\item[(ii)] $\bar\pi$ is constant on each coset in $\R/T$ on the pieces $(l, u)$ and $(f-u, f-l)$. 
\end{enumerate}
Furthermore, it follows from the additivity relations of $\pi$ and $E(\pi) \subseteq E(\bar\pi)$ that 
\begin{enumerate}
\item[(iii)] $\bar\pi(x)+\bar\pi(y)=0$ for $x, y \in (l,u)$ such that $x+y \in \{l+u, l+u-t_1, l+u-t_2\}$;
\item[(iv)] $\bar\pi(x)+\bar\pi(y)=0$ for $x \in (l,u)$, $y \in (f-u, f-l)$ such that $x+y = f$.
\end{enumerate}
We now show that $\bar\pi$ also satisfies the following condition:
\begin{enumerate}
\item[(v)] $\left|\bar\pi(x)\right| \leq s$ for all $x \in (l, u) \cup (f-u, f-l)$.
\end{enumerate}
Indeed, by (iii) and (iv), it suffices to show that for any $x \in (l, u)$, we have $\bar\pi(x) \geq -s$. Suppose, for the sake of contradiction, that there is $\bar x \in (l, u)$ such that $\bar\pi(\bar x) < -s$. Since the group $T$ is dense in $\R$, we can find $x \in (l, u)$ such that $x \in \bar x + T$ and $x$ is arbitrarily close to $(1+l - x_{39})$. Let $\delta = x - (1+l - x_{39})$. We may assume that $\delta \in (0, \frac{-s-\bar\pi(\bar x)}{c_2 - c_3})$, where $c_2$ and $c_3$ denote the slope of $\pi$ on the pieces $(l, u)$ and $(0, x_1)$, respectively. See \cite[{Table~\ref{crazy:tab:kzh_minimal_has_only_crazy_perturbation_1}}]{koeppe-zhou:crazy-perturbation} for the concrete values of the parameters. 
Let $y = 1+l -x$. Then $y = x_{39} -\delta$. It follows from (i) that $\bar\pi(y)=0$ and $\bar\pi (x+y) = \bar\pi(l)=0$. Now consider $\Delta\pi'(x,y) = \pi'(x)+\pi'(y)-\pi'(x+y)$, where
\[
\begin{array}{r@{\;}l@{\;}l}
\pi'(x)&= \bar\pi(x) + \pi(x)&=\bar\pi(x) +\pi(1+l-x_{39})+\delta c_2;\\
\pi'(y)&=\pi(y)&=\pi(x_{39}^-) - \delta c_3;\\
\pi'(x+y)&=\pi(x+y)&=\pi(l).  
\end{array}
\]
Since $x - \bar x \in T$, the condition (ii) implies that $\bar\pi(x) = \bar\pi(\bar x)$.
We have
\begin{align*}
\Delta\pi'(x,y) &= \bar\pi(\bar x)+[\pi(1+l-x_{39})+ \pi(x_{39}^-)-\pi(l)]+\delta(c_2 - c_3) \\
& =\bar\pi(\bar x) + s + \delta(c_2 - c_3)
  < 0,
\end{align*}
a contradiction to the subadditivity of $\pi'$. Therefore, $\bar\pi$ satisfies
condition~(v).

Let $F$ be a face of $\Delta\P$. Denote by $n_F \in \{0, 1, 2\}$ the number of projections $p_i(\relint(F))$ for $i=1, 2, 3$ that intersect with  $(l, u) \cup (f-u, f-l)$.  See \autoref{fig:complex_nf}.
It follows from the conditions (i) and (v) that
\[\left|\Delta\bar\pi(x,y)\right| \leq n_F \cdot s \quad\text{ for any } (x,y) \in \relint(F).\]
It can be verified computationally that, if $F\in \Delta\P$ has $n_F \neq 0$, then either 
\begin{enumerate}
\item[(a)] $\Delta\pi_F(u,v)=0$ for all $(u,v) \in \verts(F)$, or
\item[(b)] $\Delta\pi_F(u,v)\geq n_F \cdot s$ for all $(u,v) \in \verts(F)$, and the inequality is strict for at least one vertex.
\end{enumerate} 
Let $(x, y) \in [0, 1)^2$ such that  $(x,y) \not\in E(\pi)$. Then $\Delta\pi(x,y)>0$ since $\pi$ is subadditive. Consider the (unique) face $F \in \Delta\P$ such that $(x,y) \in \relint(F)$.  We will show that $\Delta\pi'(x,y) > 0$. If $n_F = 0$, then $\Delta\bar\pi(x,y)=0$, and hence $\Delta\pi'(x,y)=\Delta\pi(x,y)>0$. Now assume that $n_F \neq 0$.  Since $\Delta\pi_F$ is affine linear on $F$, $\Delta\pi(x,y)$ is a convex combination of $\{\Delta\pi_F(u, v) \st (u, v) \in \verts(F)\}$. We have $\Delta\pi(x,y) > 0$ by assumption. Thus the above case (b) applies, which implies that $\Delta\pi(x,y) > n_F \cdot s$. Hence $\Delta\pi'(x,y)=\Delta\pi(x,y)+\Delta\bar\pi(x,y) > 0$ holds when $n_F \neq 0$ as well. Therefore, $(x,y)\not\in E(\pi')$. We obtain that $E(\pi')\subseteq E(\pi)$. This, together with the assumption $E(\pi)\subseteq E(\pi')$, implies that $E(\pi)=  E(\pi')$. 

We conclude, by \autoref{lemma:weak_facet_minimal}(3), that $\pi$ is a weak facet.

\textit{Remark:} Conversely, if a $\Z$-periodic function $\bar\pi$ satisfies the conditions (i) to (v), then $\pi^\pm = \pi \pm \bar\pi$ are minimal valid functions, and $E(\pi)=E(\pi^+)=E(\pi^-)$.  

\smallbreak

(3) Let $\bar\pi = \pi_{\mathrm{lifted}} - \pi$
. Observe that $\bar\pi$ satisfies the conditions (i) to (v) in (2). 
Thus, the function $\pi_{\mathrm{lifted}}$ is minimal valid and $E(\pi_{\mathrm{lifted}})=E(\pi)$. Let $\pi'$ be a minimal valid function such that $E(\pi_{\mathrm{lifted}})\subseteq E(\pi')$. Then, as shown in (2), we have $E(\pi_{\mathrm{lifted}}) = E(\pi')$. It follows from \autoref{lemma:weak_facet_minimal}(3) that $\pi_{\mathrm{lifted}}$ is a weak facet. However, the function $\pi_{\mathrm{lifted}}$ is not a facet, since $E(\pi_{\mathrm{lifted}})=E(\pi)$ but $\pi_{\mathrm{lifted}} \neq \pi$. Next, we show that $\pi_{\mathrm{lifted}}$ is an extreme function. 

Suppose 
that $\pi_{\mathrm{lifted}}$ can be written as $\pi_{\mathrm{lifted}} = \frac12 (\pi^1+\pi^2)$, where $\pi^1, \pi^2$ are minimal valid functions. 
Then $E(\pi_{\mathrm{lifted}}) \subseteq E(\pi^1)$ and $E(\pi_{\mathrm{lifted}}) \subseteq E(\pi^2)$. 
Let $\bar\pi^1= \pi^1-\pi$ and $\bar\pi^2=\pi^2-\pi$. We have that  $E(\pi) \subseteq E(\bar\pi^1)$ and $E(\pi) \subseteq E(\bar\pi^2)$. Hence, as shown in (2), $\bar\pi^1$ and $\bar\pi^2$ satisfy the conditions (i) to (v). We will show that $\bar\pi^1=\bar\pi^2$.

For $x\not\in (l,u)\cup (f-u, f-l)$, we have $\bar\pi^i(x)=0$ ($i=1, 2$) by condition (i). It remains to prove that $\bar\pi^1(x)=\bar\pi^2(x)$ for $x \in (l,u)\cup (f-u, f-l)$. By the symmetry condition (iv), it suffices to consider $x \in (l, u)$. We distinguish three cases. If $x \in C$, then condition (iii) implies $\bar\pi^i(x)=0$ ($i=1, 2$). If $x \in C^+$, then $\bar\pi(x)=s$ by definition. Notice that $\bar\pi^1 +\bar\pi^2 = \pi^1+\pi^2-2\pi =2\pi_{\mathrm{lifted}}-2\pi = 2\bar\pi$, and that $\bar\pi^i(x) \leq s$ ($i=1, 2$) by condition (v). We have $\bar\pi^i(x)=s$ ($i=1, 2$) in this case. If $x \not\in C $ and $x \not\in C^+$, then $\bar\pi(x)=-s$, and hence $\bar\pi^i(x)=-s$ ($i=1, 2$). Therefore, $\bar\pi^1=\bar\pi^2$ and $\pi^1=\pi^2$, which proves that the function $\pi_{\mathrm{lifted}}$ is extreme.
\end{proof}

\clearpage
\providecommand\ISBN{ISBN }
\bibliographystyle{../amsabbrvurl}
\bibliography{../bib/MLFCB_bib}

\clearpage
\appendix

\section{Discontinuous Gomory--Johnson functions in the literature}
\label{appendix:discont-lit}

In our paper we investigate the notions of facets and weak
facets in the case of discontinuous functions.  This appears to be a first in
the published literature.  All papers that consider discontinuous functions
only used the notion of extreme functions. 
In particular,
Dey--Richard--Li--Miller \cite{dey1}, who were the first to consider
previously known discontinuous functions as first-class members of the
Gomory--Johnson hierarchy of valid functions, use extreme functions
exclusively; whereas \cite{dey3}, which was completed by a subset of the
authors in the same year, uses (weak) facets exclusively. 
The same is true in Dey's Ph.D. thesis \cite{Dey-thesis}:  The notion of
extreme functions is used in chapters regarding discontinuous functions;
whereas the notion of facets is used when talking about (2-row) continuous
functions.  
Dey (2016, personal communication) remembers that at that time, he and his
coauthors were aware that facets were the strongest notion and they would
strive to establish facetness of valid functions whenever possible.
However, in the excellent survey \cite{Richard-Dey-2010:50-year-survey},
facets are no longer mentioned and the exposition is in terms of extreme
functions.

\clearpage
\section{Omitted lemmas and proofs}
\label{appendix:omitted}

In the proof of \autoref{lemma:lipschitz-equiv-perturbation}, we will need the
following elementary geometric estimate. 
\begin{lemma}
\label{lemma:affine-function-min-value}
Let $F \subset [0,1]^2$ be a convex polygon with vertex set $\verts(F)$, and let $g\colon F \to \R$ be an affine linear function. Suppose that for each $v \in \verts(F)$, either $g(v) = 0$ or $g(v) \geq m$ for some $m > 0$. Let $S = \{x \in F\st g(x) = 0\}$, and assume that $S$ is nonempty. Then $g(x) \geq md(x, S)/2$ for any $x \in F$, where $d(x, S)$ denotes the Euclidean distance from $x$ to $S$.
\end{lemma}
\begin{proof}
Let $x \in F$ be arbitrary. We may write
\[ x = \sum_{v \in \verts(F)} \alpha_v v \]
for some $\alpha_v \in [0,1]$ with $\sum_v \alpha_v = 1$. By the triangle inequality,
\[ d(x, S) \leq \sum_{v \in \verts(F)} \alpha_v d(v, S). \]
For those $v \in \verts(F)$ with $g(v) = 0$, we have $v \in S$ by definition and thus $d(v,S) = 0$. Therefore,
\[ d(x, S) \leq \sum_{\substack{v \in \verts(F) \\ g(v) \geq m}} \alpha_v d(v, S) \leq 2 \sum_{\substack{v \in \verts(F) \\ g(v) \geq m}} \alpha_v . \]
Using the affine linearity of $g$, it thus follows that
\[ g(x) = \sum_{v \in \verts(F)} \alpha_v g(v) = \sum_{\substack{v \in \verts(F) \\ g(v) = 0}} \alpha_v g(v) + \sum_{\substack{v \in \verts(F) \\ g(v) \geq m}} \alpha_v g(v) \geq \frac{md(x, S)}{2}.\]
\end{proof}

\begin{proof}[of \autoref{lemma:lipschitz-equiv-perturbation}]
Let 
\begin{align*}
m := \min\{\,\Delta\pi_F(x,y) \mid {} & (x,y)\in \verts(\Delta \mathcal{P}),\, F
  \text{ is a face of } \Delta \mathcal{P} \\ & \text { such that } (x,y)\in F \text{ and } \Delta\pi_F(x,y)\neq 0\,\};
\end{align*}
Let $C$ be a positive number that is greater than the Lipschitz constant of $\bar\pi$ over the relative interior of each face of the complex $\P$, and let
\[M := \sup_{(x,y)\in \R^2} \left|\Delta\bar \pi(x,y)\right|.\]
Note that $C$ and $M$ are well defined since $\bar\pi$ is piecewise Lipschitz continuous over $\P$ and hence bounded. If $M=0$, then $\bar \pi \equiv 0$ and $\bar \pi \in \tilde\Pi^{\pi}(\R,\Z)$ holds trivially. In the following, we assume $M>0$. 
Define 
\(\epsilon := \min\big\{\frac{m}{M}, \frac{m}{8C}\big\}.\)
We also have $m > 0$, since $\pi$ is subadditive and $\Delta\pi$ is non-zero somewhere. Thus, $\epsilon>0$.
Let $\pi^+ = \pi+\epsilon\bar\pi$ and $\pi^- = \pi-\epsilon\bar\pi$. We want
to show that $\pi^{\pm}$ are minimal valid
. 

We claim that $\pi^+$ and $\pi^-$ are subadditive functions. Let $(x, y) \in [0,1)^2$. Let $F$ be a face of $\Delta \P$ such that $(x, y) \in F$. We will show that $\Delta\pi_F^{\pm}(x, y) \geq 0$. 
First, assume $\Delta\pi_F(x, y)=0$. It follows from $E_F(\pi)\subseteq E_F(\bar\pi)$ that $\Delta\bar\pi_F(x, y)=0$. Therefore, $\Delta\pi_F^{\pm}(x, y) = 0$.
Next, assume $\Delta\pi_F(x, y) \neq 0$. Consider $S=\{(u,v) \in F \mid \Delta\pi_F(u,v)=0\}$, which is a closed set since $\Delta\pi_F$ is continuous over the face $F$. 
If $S=\emptyset$, then $\Delta\pi_F(u,v) \geq m$ for any $(u,v) \in \verts(F)$. We have $\Delta\pi_F(x,y)\geq m$ by the fact that $\Delta\pi_F$ is affine over $F$. 
Hence, in this case, 
\[\Delta\pi_F^{\pm}(x, y) 
  \geq \Delta\pi_F(x, y) -\epsilon\left| \Delta\bar\pi_F(x,y)\right| \geq m - \tfrac{m}{M}M \geq 0.\]
Now consider the case $S\neq \emptyset$. 
Let $d$ denote the Euclidean distance from $(x,y)$ to $S$. Since $S$ is a closed set, there exists $(x',y') \in S$ such that $(x-x')^2+(y-y')^2 = d^2$. 
Let $I = p_1(F), J = p_2(F)$ and $K = p_3(F)$. Then $x, x' \in I, \; y, y' \in J$ and $x+y, x'+y' \in K$.
It follows from $E_F(\pi)\subseteq E_F(\bar\pi)$ and $\Delta\pi_F(x', y')=0$ that $\Delta\bar\pi_F(x', y')=0$.
Therefore,
\begin{align*}
  \Delta\bar\pi_F(x,y) &= \Delta\bar\pi_F(x,y) -\Delta\bar\pi_F(x',y') \\
                       &= \bar\pi_I(x) - \bar\pi_I(x')+\bar\pi_J(y) - \bar\pi_J(y')+\bar\pi_K(x+y)-\bar\pi_K(x'+y').
\end{align*}
Since $\bar\pi$ is Lipschitz continuous over $\relint(I), \relint(J)$ and $\relint(K)$, we have that
\begin{align*}
\left|\bar\pi_I(x) - \bar\pi_I(x')\right| &\leq  C \left| x -x' \right| \leq Cd; \\
\left|\bar\pi_J(y) - \bar\pi_J(y')\right| &\leq  C \left| y -y' \right| \leq Cd; \\
\left|\bar\pi_K(x+y) - \bar\pi_K(x'+y')\right| &\leq  C \left| x+y -x'-y' \right| \leq 2Cd.
\end{align*}
Hence $\left|\Delta\bar\pi_F(x,y)\right| \leq 4Cd$.
Applying a geometric estimate
(\autoref{lemma:affine-function-min-value} with $g=\Delta\pi_F$) shows that
$\Delta\pi_F(x,y)\geq \frac{md}{2}$. 
Therefore, in the case where $S \neq \emptyset$,
\begin{align*}
\Delta\pi_F^{\pm}(x, y) &= \Delta\pi_F(x, y)  \pm \epsilon\Delta\bar\pi_F(x,y) \\
& \geq \Delta\pi_F(x, y) -\epsilon\left| \Delta\bar\pi_F(x,y)\right| 
\geq \frac{md}{2} - \frac{m}{8C}(4Cd) = 0.
\end{align*}

We showed that $\pi^{\pm}$ are subadditive. Since $\bar\pi \in \bar\Pi^E(\R,\Z)$, we have $\pi^{\pm}(0)=\pi(0)=0$ and $\pi^{\pm}(f)=\pi(f) =1$. The last result along with $E(\pi)\subseteq E(\bar\pi)$ imply that $\pi^+(x)+ \pi^+(y)=\pi^-(x)+ \pi^-(y) = 1$ if $x + y \equiv f \pmod 1$. The functions $\pi^\pm$ are non-negative. Indeed, suppose that $\pi^+(x)<0$ for some $x\in\R$, then it follows from the subadditivity that $\pi^+(nx)\leq n\pi^+(x)$ for any $n \in \Z_+$, which is a contradiction to the boundedness of $\pi^+$.

Thus, $\pi^\pm$ are minimal valid functions. We conclude that $\bar \pi \in \tilde\Pi^{\pi}(\R,\Z)$.
\end{proof}

\clearpage
\section{Reformulation of the Facet Theorem using perturbation functions}
\label{appendix:facet-theorem-with-perturbation}

In \cite[page 25, section 3.6]{igp_survey}, the Facet Theorem is reformulated in terms of
perturbation functions as follows:
\begin{quote}
  If $\pi$ is not a facet, then 
  there exists a non-zero $\bar \pi \in \bar \Pi^{E(\pi)}(\R,\Z)$ such that $\pi' = \pi + \bar \pi$ is a minimal valid function.   
\end{quote}
It then cautions that this last statement is not an ``if and only if'' statement.
We now prove that the following ``if and only if'' version holds.

\begin{lemma}
\label{lemma:facet_no_bar_pi}
A minimal valid function $\pi$ is a facet if and only if there is no non-zero $\bar{\pi} \in \bar\Pi^E(\R,\Z)$, where $E=E(\pi)$, such that $\pi+\bar\pi$ is minimal valid.
\end{lemma}
\begin{proof}
Let $\pi$ be a minimal valid function. 

Assume that $\pi$ is a facet. Let $\bar{\pi} \in \bar\Pi^E(\R,\Z)$ where $E=E(\pi)$ such that $\pi' = \pi+\bar\pi$ is minimal valid. It is clear that $E(\pi) \subseteq E(\pi')$. By \autoref{lemma:facet_theorem}, $\pi' = \pi$. Thus, $\bar\pi \equiv 0$.

Assume there is no non-zero $\bar{\pi} \in \bar\Pi^E(\R,\Z)$, where $E=E(\pi)$, such that $\pi+\bar\pi$ is minimal valid. Let $\pi'$ be a minimal valid function such that $E(\pi) \subseteq E(\pi')$. Consider $\bar\pi = \pi' -\pi$. We have that $\bar{\pi} \in \bar\Pi^E(\R,\Z)$ and that $\pi+\bar\pi = \pi'$ is minimal valid. Then $\bar\pi \equiv 0$ by the assumption. Hence, $\pi' = \pi$. It follows from \autoref{lemma:facet_theorem} that $\pi$ is a facet.
\end{proof}


\clearpage
\section{Additional illustrations}

\begin{figure}[h]
\centering
\begin{minipage}{.49\textwidth}
\centering
\includegraphics[width=.8\linewidth]{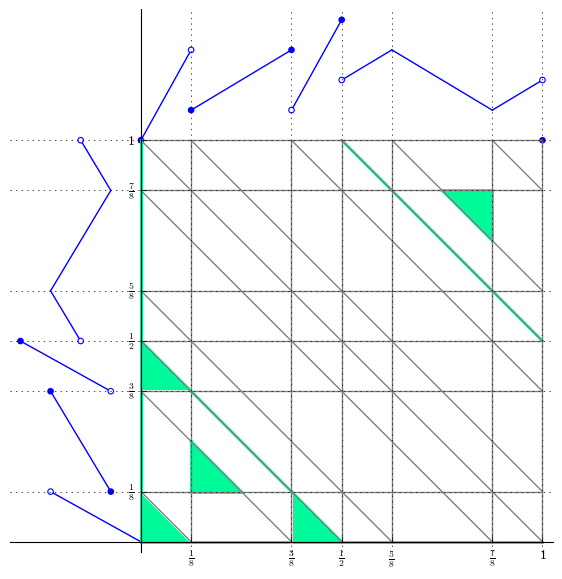}
\end{minipage}
\begin{minipage}{.49\textwidth}
\centering
\includegraphics[width=.8\linewidth]{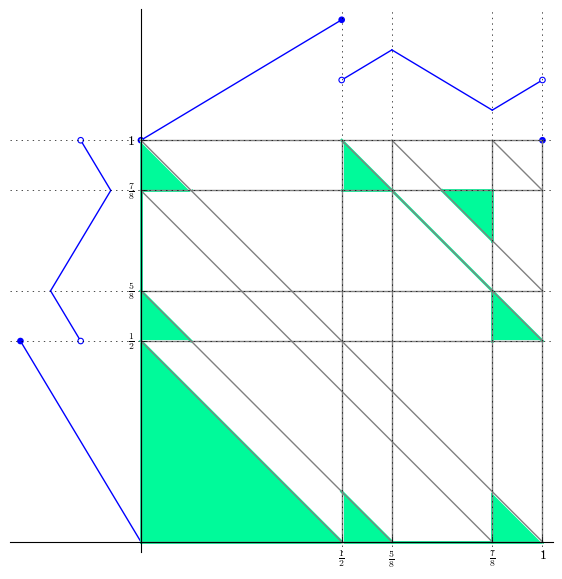}
\end{minipage}
\caption{Two diagrams of functions \sage{h} (\textit{blue graphs on the top and the left}) and polyhedral complexes $\Delta\P$ (\textit{gray solid lines}) with additive domains $E(\sage{h})$ (\textit{shaded in green}), 
as plotted by the command
\sage{plot\underscore{}2d\underscore{}diagram\underscore{}additive\underscore{}domain\underscore{}sans\underscore{}limits(h)}. (\textit{Left})
$\sage{h = hildebrand\_discont\_3\_slope\_1()}=\pi$. (\textit{Right}) \sage{h}
$=\pi'$ from the proof of \autoref{lemma:discontinuous_examples}\,(2).}
\label{fig:simple_E_pi_extreme_not_facet}
\end{figure}

\begin{figure}[h!]
\centering
\includegraphics[width=.6\linewidth]{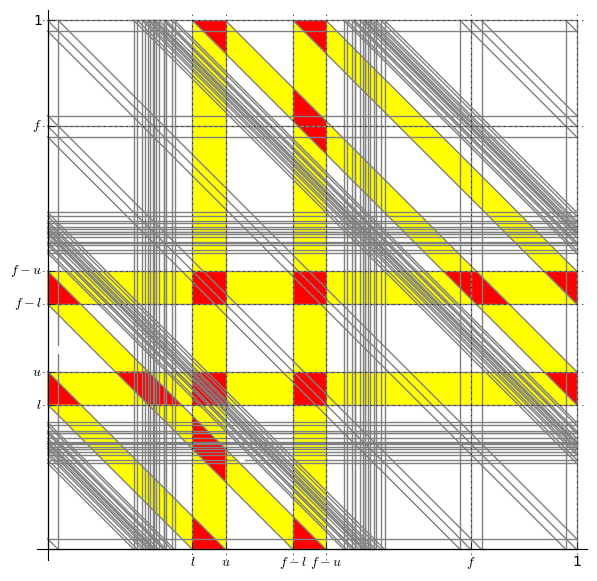}
\caption{Diagram of the polyhedral complex $\Delta\P$ of the function $\pi= \sage{kzh\_minimal\_has\_only\_crazy\_perturbation\_1()}$, where faces $F$ are color-coded according to the values $n_F$: $n_F=0$ (\emph{white}), $n_F=1$ (\emph{yellow}), $n_F=2$ (\emph{red}).}
\label{fig:complex_nf}
\end{figure}
            
\end{document}